\documentclass[10pt]{article} 

\usepackage[leqno]{amsmath}
\usepackage{amssymb}
\usepackage{amsfonts} 
\usepackage{amstext} 
\usepackage{amsthm} 
\usepackage{euscript} 
\usepackage{eufrak} 
\makeatletter
\renewcommand{\theenumi}{\roman{enumi}}

\renewcommand{\p@enumii}{\theenumi--}
\makeatother

\newcommand{\Z}{\ensuremath{\mathbb{Z}}}
\newcommand{\N}{\ensuremath{\mathbb{N}}}
\newcommand{\Q}{\ensuremath{\mathbb{Q}}}
\newcommand{\R}{\ensuremath{\mathbb{R}}} 
\newcommand{\C}{\ensuremath{\mathbb{C}}}

\newcommand{\cH}{\ensuremath{\mathcal{H}}}

\newcommand{\cS}{\ensuremath{\mathcal{S}}}
\newcommand{\cT}{\ensuremath{\mathcal{T}}}

\newcommand{\lac}{\ensuremath{\lambda_{\text{can}}}}

\newcommand{\Ham}{\ensuremath{\text{Ham}}} 
\newcommand{\Hom}{\ensuremath{\text{Hom}}}

\begin{document} 
\sloppy
\begin{center} 
{\Large \bf Monodromy and isotopy of monotone Lagrangian tori} \\ 

\vspace{.3in} 
Mei-Lin Yau \footnote{Research 
Supported in part by National Science Council grants 
95-2115-M-008-012-MY2 and 97-2115-M-008-009-. 
 

}

  \end{center} 
  
\vspace{.2in} 
\begin{abstract} 
We define new Hamiltonian isotopy invariants for a 2-dimensional monotone Lagrangian torus embedded in a symplectic 4-manifold. 
We show that, in the standard symplectic $\R^4$, these invariants distinguish a monotone Clifford torus from a Chekanov torus. 
\end{abstract}

\newtheorem{theo}{Theorem}[section]
\newtheorem{cond}[theo]{Condition}
\newtheorem{defn}[theo]{Definition}
\newtheorem{exam}[theo]{Example}
\newtheorem{lem}[theo]{Lemma}
\newtheorem{cor}[theo]{Corollary}

\newtheorem{prop}[theo]{Proposition}
\newtheorem{rem}[theo]{Remark}
\newtheorem{notn}[theo]{Notation}
\newtheorem{fact}[theo]{Fact}
\newtheorem{ques}[theo]{Question}

%
%\section{Introduction} 
%

\section{Introduction} 

This article concerns the Hamiltonian isotopy problem of {\em monotone} Lagrangian tori, which is a special case of the Lagrangian knot problem as formulated by Eliashberg and Polterovich in \cite{EP3}.  Two Lagrangian tori $L_0,L_1$ embedded in a symplectic 4-manifold $(M,\omega)$ are said to be {\em Hamiltonian isotopic} if there exists a smooth isotopy of Hamiltonian diffeomorphisms 
$\phi_t\in \Ham(M)$ with compact support, $t\in [0,1]$, $\phi_0=id$, such that $\phi_1(L_0)=L_1$.

We study the monodromy group $\cH_L$ of the Hamiltonian self-isotopies of a monotone Lagrangian torus $L$. We define two new Hamiltonian isotopy invariants for $L$: the {\em twist number} $t(L)\in \N\cup\{ 0\}$ and the {\em spectrum} $s(L)\in \N\cup\{ 0\}$ of $L$. The twist number $t(L)$ is related to 
Dehn twists along a embedded curve with 0 Maslov number, while the spectrum $s(L)$ involves 
the Maslov number of the primitive integral 1-eigenvectors of involutions of $\cH_L$ (see Definition \ref{mf} and Proposition \ref{012}).

We then apply our construction to  the cases when $L\subset \R^4$ is either a monotone Clifford torus 
$T_{b,b}$ or a Chekanov torus $T'_{0,b}$ (called {\em special torus} in \cite{C1}) in the standard symplectic 4-space $\R^4$. We obtain the following:

\begin{theo}  \label{main} % theorem \label{main}
Let $b>0$. Let $\cH_b$ denote the Hamiltonian monodromy group of $T_{b,b}$, $\cH'_b$ the Hamiltonian monodromy group of $T'_{0,b}$. Then 
$\cH_b\cong \Z_2\cong\cH'_b$ as abstract groups and hence $t(T_{b,b})=0=t(T'_{0,b})$. 
However, 
\[ 
s(T_{b,b})=2, \quad s(T'_{0,b})=1. 
\] 
Hence $T_{b,b}$ and $T'_{0,b}$ are not Hamiltonian isotopic in $\R^4$. 
\end{theo} 
Thus our approach provides a new way to distinguish  $T_{b,b}$ from  
$T'_{0,b}$ up to  Hamiltonian isotopy.

It should be pointed out that, the Hamiltonian non-isotopy between $T_{b,b}$ and $T'_{0,b}$ has been proved by Chekanov \cite{C1,C2}. Chekanov gave two proofs on this result. The first proof in \cite{C1}  utilized the symplectic 
capacities introduced by Ekeland and Hofer \cite{EkH1,EkH2}, whilst the second 
proof \cite{C2} employed  pseudoholomorphic curves with boundaries \cite{G}. 
Indeed, Chekanov dealt with Clifford tori and Chekanov tori in symplectic $\R^{2n}$ with $n\geq 2$, and completely classified such tori in all $\R^{2n}$. 

Comparing with Chekanov's proofs, our approach is more algebraic in nature, and 
seemingly simpler and more elementary.  For technical simplicity, we do not venture into  higher dimensional cases here. However, we expect that, given suitable generalization, invariants similar to $t(L)$ and $s(L)$ can be defined for monotone Lagrangian 
tori of general dimensions. We hope to come back to this topic later. 

This paper is organized as follows: In Section \ref{mono} we start with some conditions on symplectic 4-manifolds, in order for the monotonicity of a Lagrangian 
torus $L$ to be well-defined. We then proceed to define the Hamiltonian monodromy group $\cH$ of a monotone torus $L$ and study its properties. Then follows the definition of the invariants $t(L),s(L)$. 
In Section \ref{TT'} we determine the Hamiltonian monodromy group as well as  
the values of $t(L),s(L)$ for  $L=T_{b,b}$ (Lemma \ref{Tb}) and $L=T'_{0,b}$ 
(Lemma \ref{T'b}). We end up this note with several open questions.

%
%\section{Lagrangian self-isotopy and monodromy}  \label{mono}  
%

\section{Hamiltonian monodromy of monotone Lagrangian tori}  \label{mono}

Let $L\overset{\iota}{\hookrightarrow}M$ be an embedded Lagrangian torus in a symplectic 4-manifold $(M,\omega)$. One can endow the tangent bundle of $M$ with an almost complex structure compatible with $\omega$. This turns $TM$ into a complex vector bundle of which the Chern classes depend only on $\omega$. 
From now on, unless otherwise mentioned, we assume that $M$ satisfies the following two conditions that (i) the first Chern class $c_1(M)=c_1(TM)\in H^2(M,\Z)$ vanishes and (ii) $H^1(M,\R)=0$. 

That $c_1(M)=0$ ensures that 
the Maslov class $\mu\in H^1(L,\Z)$ is well-defined. There is a unique integer $m_L\geq 0$ such that $\mu(H_1(L,\Z))=m_L\Z$. We call $m_L$ the {\em divisibility} of  $\mu$ on $L$. 

Near $L$ the symplectic form $\omega$ is exact, i.e., there exits a 1-form $\lambda$ defined on a tubular neighborhood $U_L$ of $L$ oh which  $\omega=d\lambda$. 
The pull-back 1-form $\iota^*\lambda\in \Omega^1(L)$ is closed, we 
denote its cohomology class in $H^1(L,\R)$ as $\alpha$. The class $\alpha$ 
is independent of the choice of $\lambda$ due to the assumption that $H^1(M,\R)=0$.

\begin{defn} 
{\rm 
Assume that $\mu\neq 0$ and 
$\alpha\neq 0$. Then $L$ is {\em monotone} if $\alpha=c\mu$ for some $c\in\R\setminus \{ 0\}$. 
}
\end{defn} 
Let $\Ham(M,L)$ denote the group of all symplectomorphisms $\phi:(M,\omega) \to (M,\omega) $ such 
that  $\phi(L)=L$ and $\phi$ is the time one map of some time dependent Hamiltonian vector field on $M$, and the vector field has compact support. 

A map $\phi\in\Ham(M,L)$ 
induces  an isomorphism on $H^1(L,\Z)\cong\Z^2$ and hence on $H^1(L,\R)$, preserving both 
$\mu$ and $\alpha$. If $\mu$ and $\alpha$ are $\R$-linearly independent then 
$\phi^*=id$ on both $H^1(L,\Z)$ and $H^1(L,\R)$. {\em Below we consider the monotone case only}. For $\phi\in\Ham(M,L)$ we call the induced isomorphism  
$\phi_*:H_1(L,\Z)\to H_1(L,\Z)$ the {\em monodromy} of $\phi$.

\begin{defn} 
{\rm 
We define the {\em Hamiltonian monodromy group} of $L\subset M$ (or $\Ham(M,L)$) to be 
\[ 
\cH_L=\{ \phi_*\in \text{Isom}(H_1(L,\Z))\mid \phi\in \Ham(M,L)\}. 
\] 
Consider the subgroup of $\text{Isom}(H_1(L,\Z))$: 
\[ 
G_\mu:=\{ g\in \text{Isom}(H_1(L,\Z))\mid \mu\circ g=\mu\}. 
\] 
Clearly $\cH_L$ is a subgroup of $G_\mu$. 
}
\end{defn}

If we fix a basis for $H_1(L,\Z)$ then the group $\text{Isom}(H_1(L,\Z))$ is identified with the integral general linear group 
\[ 
GL(2,\Z)=\Big\{ \begin{pmatrix} a& b \\ c & d\end{pmatrix}\mid a,b,c,d\in \Z, \ 
ad-bc=\pm1\Big\} . 
\] 
In the following we fix a basis $\gamma,\sigma$ for $H_1(L,\Z)$ so that 
$\mu(\gamma)=\min\{ \mu(\eta)\mid \eta\in H_1(L,\Z), \ \mu(\eta)>0\}$, 
and $\sigma\in \ker\mu$.  We identify an element $p\gamma+q\sigma$ of 
$H_1(L,\Z)$ with the column vector $(p,q)^t\in \Z^2$. Then $\mu$ is identified 
with the row vector $(m_L, 0)\in \Hom(H_1(L,\Z),\Z)\cong \Z^2$ where $m_L=\mu(\gamma)$, and 
$\phi_*$ is identified with an element of $GL(2,\Z)$ such that 
$\mu\circ \phi_*=\mu$. We have the following:

\begin{prop} \label{Gmu}    % prop \label{Gmu} 
The group $G_\mu:=\{ g\in GL(2,\Z)\mid \mu \circ g=\mu\}$ is isomorphic 
to the infinite dihedral group $D_\infty:=\langle f_0,f_1\mid f_0^2=e=f_1^2\rangle=\Z_2*\Z_2$ (see \cite{Hum}), 
where $e$ denote the identity element. 
\end{prop}

\begin{proof} 
A direct computation (identifying $\mu$ with $(m_L,0)$) shows that, with respect to the basis $\{ \gamma, \sigma\}$ for $H_1(L,\Z)$  as described above,  
\[ 
G_\mu=\Big\{ \begin{pmatrix} 1& 0 \\ k & \pm 1\end{pmatrix}\mid k\in \Z \ 
\Big\}  . 
\] 
Let $f_k:=\begin{pmatrix} 1& 0 \\ k & -1\end{pmatrix}$ and 
$g_k:=\begin{pmatrix} 1& 0 \\ k &  1\end{pmatrix}$. We have 
\begin{enumerate} 
\item $f_k^2=e$ and $g_k=g_1^k$ for $k\in \Z$,  $g_k$ has infinite order for $k\neq 0$,   
\item $g_1=f_1f_0$ and hence $g_{-1}=g^{-1}_1=f_0f_1$,   
\item $f_{k+1}=g_kf_1=(f_1f_0)^kf_1=f_1g_{-k}$ and $f_{-k}=g_{-k}f_0=(f_0f_1)^kf_0=f_0g_k$ for $k\geq 0$. 
\end{enumerate} 
Readers can check that $G_\mu$ is indeed freely generated by the two elements 
$f_0$ and $f_1$ of order 2, hence $G_\mu$ is isomorphic to the infinite dihedral group. 
\end{proof} 

\begin{rem} \label{geom}  % remark \label{geom}
{\rm 
Geometrically $g_k$ is a ($\pm k$)-Dehn twist along the curve representing the class $\sigma$. 
The $\pm$-sign ambiguity is to the non-uniqueness of $\sigma$ which is unique only up to signs. 
} 
\end{rem}

\begin{rem}  \label{tilde}   % remark \label{tilde} 
{\rm 
Sometimes it is more convenient to consider a different basis $\{ \gamma, \gamma':=\gamma+\sigma\}$ for $H_1(L,\Z)$. The transformation matrix from basis $\{ \gamma, \gamma'\}$ to basis $\{ \gamma,\sigma\}$ is $T:=\begin{pmatrix}
1 & 1\\ 0 & 1\end{pmatrix}$. We use $\tilde{\ \   }$ to denote the matrix representation 
with respect to the basis $\{ \gamma,\gamma'\}$. Then for $k\in \Z$, 
\[ 
\tilde{f}_k:=T^{-1}f_kT=\begin{pmatrix} 1-k & 2-k \\ k & k-1\end{pmatrix}, \quad 
\tilde{g}_k:=T^{-1}g_kT=\begin{pmatrix} 1-k & -k \\ k & 1+k\end{pmatrix}.
\] 
} 
\end{rem}

\begin{lem} \label{fkfl}   %\label{fkfl}
Let $f_k,g_k$ be as in the proof of Proposition \ref{Gmu}. Then $f_kf_l=g_{k-l}$ and hence $f_k=g_{k-l}f_l$ for $k,l\in \Z$. 
\end{lem} 

\begin{proof} 
By applying (iii) from the proof of Proposition \ref{Gmu}, we have $f_kf_0=g_k$ and $f_0f_l=g_{-l}$ for $k,l\in \Z$. Hence $f_kf_l=(f_kf_0)(f_0f_l)=g_kg_{-l}=g_{k-l}$. i.e., $f_k=g_{k-l}f_l$, for $k,l\in \Z$. 
\end{proof}

Lemma \ref{fkfl} in particular implies a different (but well-known) way of expressing $G_\mu\cong D_\infty$ by  generators and  relations: 
\begin{equation} 
G_\mu\cong D_\infty=\langle f, g \mid f^2=e, \  fgf=g^{-1}\rangle .  
\end{equation}   

\begin{defn} 
{\rm 
An element $f\in G_\mu$ is called an {\em  involution} if $f^2=e$ and $f\neq e$, a {\em twist} if 
$\text{ord}(f)=\infty$. 
}
\end{defn} 

Then, with  the matrix representations adapted in the proof of Proposition \ref{Gmu}, 
$f_k\in G_\mu$ is an involution for any $k\in\Z$, and $g_k$ a twist for any $k\in \Z\setminus \{ 0\}$. 

\begin{lem}  \label{H}  % lemma  \label{H} 
Let $\cH\subset G_\mu$ be a subgroup of $G_\mu$. We adapt the notations $f_k,g_k$ from the proof of Proposition \ref{Gmu}. Then we have the following classification of $\cH$: 
\begin{enumerate} 
\item $\cH=\{e\}$. 
\item If $\cH\neq \{e\}$ and $\cH$ has no involutions, then $\cH=\langle g_d\rangle \cong \Z$ for some $d\in \N$. 
\item If $\cH$ contains exactly one involution say, $f_k$,  then $\cH=\langle f_k\rangle\cong \Z_2$. 
\item If $\cH$ contains two involutions, then there exist some $k, d\in \N$ such that 
$\cH=\langle f_k, g_d\rangle\cong D_\infty$. 
\end{enumerate} 
\end{lem} 

\begin{proof} 
For computational simplicity, we adapt the matrix representations and notations from the proof of 
Proposition \ref{Gmu}. 

Assume that $\cH$ has no involutions. Then $\cH$ is a subgroup of the free abelian group $\langle g_1\rangle$ generated by $g_1\in G_\mu$. Hence $\cH=\langle g_d\rangle$ for some $d\in \N$, $\cH\cong \Z$. 

Now suppose that $\cH$ contains exactly one involution say, $f_k$, for some  $k\in \Z$. We claim that $\cH=\{ e,f_k\}$ and hence $\cH\cong \Z_2$. Otherwise, we would have $g_l\in \cH$ for some $l\in \Z$, $l\neq 0$. But then $f_{k+l}=g_lf_k\in \cH$ by Lemma \ref{fkfl}, which contradicts with our assumption on $\cH$. Hence  $\cH=\{ e,f_k\}\cong \Z_2$ for some $k\in \Z$. 

Finally, assume that $\cH$ contains more than one involution. Let 
\[ 
d:=\min \{ |k-l|\mid f_k,f_l\in \cH, \ f_k\neq f_l\}\in \N 
\] 
Fix a number $k\in \Z$ such that $f_k,f_{k+d}\in \cH$. We claim that $\cH$ is generated by $f_k$ and $f_{k+d}$. More precisely we will show that  
\begin{enumerate} 
\item $g_m\in \cH$ iff $d$ divides $m$, 
\item $f_n\in \cH$ iff $d$ divides $n-k$.  
\end{enumerate} 
Assume that $m=ds$ for some $s\in \Z$. Then $g_m=g_{ds}=g_d^s=(f_{k+d}f_k)^s\in \cH$. Similarly, assume that $n=k+dr$ for some $r\in \Z$. then $f_n=f_{k+dr}=g_{dr}f_k\in\cH$. Conversely, assume that $g_m\in \cH$ fro some $m=ds+r\in \Z$ with 
$d,s,r\in \Z$, $0<r<d$. Then $g_r=g_mg_{-ds}\in \cH$ and $f_{k+r}=g_rf_k\in \cH$. But then $|(k+r)-k|<d$,  which contradicts with the minimality of $d$. Similarly, 
if $f_n\in \cH$ for some $n\in \Z$, then $g_{n-k}=f_nf_k\in \cH$, which implies that $d$ divides $n-k$. This verifies the claim. So $\cH$ is freely generated by $f_k,f_{k+d}$,  two elements of order 2. Hence $\cH$ is isomorphic to $D_\infty$. Since $g_d=f_{k+d}f_k$, $\cH$ is also generated by $f_k$ and $g_d$. 
Finally, replacing $f_k$ by $f_{k+sd}=g_{sd}f_k$ for some  $s\in \N$ large enough if necessary, we may 
assume that $k>0$ 
This completes the proof. 
\end{proof}

\begin{defn}   \label{mf}   % definition   \label{mf} 
{\rm 
Let  $f\in \cH_L$  be  an involution. Then 
$f$ is diagonalizable with eigenvalues $1,-1$. Let $\text{Fix}(f):=\{ \eta\in H_1(L,\Z)\mid f(\eta)=\eta\}$.  
There is a unique element $\eta_f\in \text{Fix}(f)$ such that $\text{Fix}(f)$ is generated by $\eta_f$ and 
$\mu(\eta_f)>0$. Define 
\[ 
m_f:=\mu(\eta_f)/m_L\in \N. 
\] 
}
\end{defn} 

Recall that $\sigma\in H_1(L,\Z)$ denotes a fixed primitive class with $\mu(\sigma)=0$. 

\begin{prop} \label{012}  % prop \label{012}
Let $m_f$ be as defined above.  Then $m_f=1$ or $2$. Moreover, if we fix an arbitrary  basis for $H_1(L,\Z)$ and represent $\eta_f,\sigma$ respectively as the first and second column vectors of an 
integral $2\times 2$ matrix $A$, then $m_f=|\det(A)|$ the absolute value of the determinant 
of $A$. 
\end{prop}

\begin{proof}
Fix a basis for $H_1(L,\Z)\cong \Z^2$ and express  $\eta_f=\begin{pmatrix} a\\ c\end{pmatrix}$ 
and $\sigma=\begin{pmatrix} b\\ d\end{pmatrix}$ as column vectors with respect to the basis. 
Let $m:=\det A=ad-bc=\det \begin{pmatrix} a & b \\ c & d\end{pmatrix}\in \Z\setminus \{ 0\}$. 
Now in matrix form 
\[ 
f=\begin{pmatrix} a & b \\ c & d\end{pmatrix} \begin{pmatrix} 1& 0 \\ 0 & -1\end{pmatrix} 
\begin{pmatrix} a & b \\ c & d\end{pmatrix}^{-1} 
=\begin{pmatrix} 1+\frac{2bc}{m} & -\frac{2ab}{m} \\ \frac{2cd}{m} & 1-\frac{2ad}{m}\end{pmatrix}\in GL(2,\Z)
\] 
So $\frac{ab}{m}, \frac{ad}{m},\frac{cb}{m},\frac{cd}{m}\in\frac{1}{2}\Z$.  Note that 
$b,d$ are coprime, so there exist $r_1,r_2\in \Z$ such that $r_1b+r_2d=1$. Then 
$r_1\cdot \frac{ba}{m}+r_2\cdot \frac{da}{m}=\frac{a}{m}\in\frac{1}{2}\Z$ and 
$r_1\cdot \frac{bc}{m}+r_2\cdot \frac{dc}{m}=\frac{c}{m}\in\frac{1}{2}\Z$. Since 
$a,c$ are coprime we must have $m\mid 2$.

Elements of $H_1(L,\Z)$ are represented as column vectors with integral coefficients. Let $n_1,n_2\in \Z$ so that 
$n_1m_L=\mu\Big(\begin{pmatrix} 1\\ 0\end{pmatrix}\Big)$ and $n_2m_L=\mu\Big(\begin{pmatrix} 0\\ 1\end{pmatrix}\Big)$. By definition we have 
$\mu\Big(\begin{pmatrix} a\\ c\end{pmatrix}\Big)=m_fm_L$ and 
$\mu\Big(\begin{pmatrix} b\\ d\end{pmatrix}\Big)=0$. So 
\[ 
\begin{pmatrix} a & c\\ b & d\end{pmatrix} \begin{pmatrix} n_1\\ n_2\end{pmatrix} =
\begin{pmatrix} m_f\\ 0\end{pmatrix}, \quad \text{i.e.,} \quad 
\begin{pmatrix} n_1\\ n_2\end{pmatrix} =\frac{m_f}{m}\begin{pmatrix} d \\ -b\end{pmatrix} . 
\] 
Note that $n_1,n_2$ are coprime since $\mu(H_1(L,\Z))=m_L\Z$. 
Also, $0=\mu\Big(\begin{pmatrix} b\\ d\end{pmatrix}\Big)=bn_1+dn_2$, so we have 
$\begin{pmatrix} n_1\\ n_2\end{pmatrix} =\pm \begin{pmatrix} d\\ -b\end{pmatrix}$, hence 
$m_f=|m|=|\det(A)|$ which equals 1 or 2. 

Now if we choose another basis for $H_1(L,\Z)$ and correspondingly represent the ordered pair 
$(\eta_f,\sigma)$ by a matrix $A'$. Then $A'=BA$ for some $B\in GL(2,\Z)$, hence 
$|\det(A')|=|\det(A)|=m_f$. This completes the proof. 
\end{proof}

\begin{exam} \label{f0f1}  % example  \label{f0f1}
Let  $f_k=\begin{pmatrix} 1 & 0\\ k & -1\end{pmatrix}$ be as defined in the proof of Proposition \ref{Gmu}. Then $m_{f_k}=\begin{cases} 1 & \text{ if $k$ is even}, \\ 2 & \text{ if $k$ is odd}. \end{cases}$ 
\end{exam} 
This is a straightforward computation. 
Let $f=f_k$. 
Write $\eta_f=\begin{pmatrix}a \\ b\end{pmatrix}$ then $\eta_f=f(\eta_f)=\begin{pmatrix} a \\ ak-b
\end{pmatrix}$. Since $a,b$ are coprime we have that $a=1$ (the positive sign follows from 
 $\mu(\eta_f)>0$) provided that $b=0$. In this case we obtain that $\mu(\eta_f)=m_L$ and hence 
 $m_f=1$. 
 
 If $b\neq 0$ then, since $ak=2b$ we have $k\neq 0$ and $\begin{pmatrix}a \\ b\end{pmatrix}
 =t\cdot \begin{pmatrix}2 \\ k \end{pmatrix}$ for some $t\in \Q$. It follows that 
 \begin{enumerate} 
 \item $t=1$ if $k$ is odd, then $\mu(\eta_f)=2m_L$, $m_f=2$; 
 \item $t=\frac{1}{2}$ if $k$ is even and nonzero, then $\mu(\eta_f)=m_L$, $m_f=1$. 
 \end{enumerate} 
This competes the computation.

From the perspective of the Maslov class $\mu$, involutions of $G_\mu$ fall into two different types 
according to their values of $m_f$. Thus in addition to the group type of $\cH_L$ the Hamiltonian monodromy group of $L$ as listed in Lemma \ref{H}, $m_f$ can be used to construct further invariants for $\cH_L$ provided that some element of $\cH$ is an involution.  Also, if $\cH_L$ contains some twist elements, the  twist number can also be 
defined for $\cH_L$. 
Below we define the new invariants for $L$.

\begin{defn}  \label{st}   % definition  \label{st}  
{\rm 
Let  $L$ be monotone and $\cH_L$ its Hamiltonian monodromy group. We adapt the notations $f_k,g_k$ from the proof of Proposition \ref{Gmu}. 

Let $\cT\subset \cH$ be the subset of all twists of $\cH$. 
The {\em twist number} of $\cH_L$ is defined to be 
\[ 
t(L):=\begin{cases} d=\min \{ k>0\mid g_k\in \cT\} & \text{ if $\cT\neq \emptyset$}, \\ 0 & \text{ if $\cT= \emptyset$}. 
\end{cases} 
\] 
Let $\cS\subset \cH$ denote the subset of all involutions of $\cH$. 
The {\em spectrum} of $\cH_L$ is defined to be 
\[ 
s(L):=\begin{cases} \min \{ m_f\mid f\in \cS\} & \text{ if $\cS\neq \emptyset$}, \\ 0 & \text{ if $\cS= \emptyset$}. 
\end{cases} 
\] 
} 
\end{defn}

In particular, if $\cH_L\cong \Z_2$, then $t(L)=0$, and $s(L)=1$ or $2$. 

\begin{lem}  
The numbers  $t(L),s(L)$ are invariants of monotone Lagrangian torus $L$ up to Hamiltonian isotopies. 
\end{lem} 

\begin{proof} 
Given a pair of Hamiltonian isotopic monotone Lagrangian tori $L_0,L_1$ and let
 $L_t:=\phi_t(L_0)$, $t\in [0,1]$, be a Hamiltonian isotopy between $L_0$ and $L_1$. 
 Here $\phi_t$ is the time $t$ map of a time dependent Hamiltonian vector field. The $L_t$ 
 is monotonic for all $t\in [0,1]$. Clearly $\phi_t^{-1}\circ \Ham(M,L_t)\circ \phi_t=\Ham(M,L_0)$ for 
 all $t$ and hence $\phi_t^*\cH_t=\cH_0$ where $\cH_t$ is the Hamiltonian monodromy group of $L_t$. 
 By continuity we have $t(L_t)=t(L_0)$ and  $s(L_t)=s(L_0)$ for all $t\in [0,1]$. This completes the proof. 
 \end{proof}

%
%\section{Examples in $\R^4$: Clifford tori and Chekanov tori} \label{TT'} 
%

\section{Examples in $\R^4$: Clifford tori and Chekanov tori} \label{TT'}

\noindent
{\bf Basic properties of Lagrangian tori in $\R^4$.} 
\ Let $M=\R^4$ with the standard symplectic structure $\omega=\sum_{j=1}^2dx_j\wedge dy_j$. Readers can check that $c_1(\R^4)=0$ and $H^1(\R^4,\R)=0$. Let $\lambda$ denote a primitive of $\omega$, 
$d\lambda=\omega$. Let $L\overset{\iota}{\hookrightarrow} \R^4$ be an 
embedded torus. Using pseudoholomorphic curves, Gromov \cite{G} showed $L$ is {\em not exact}, i.e., 
the closed 1-form $\iota^*\lambda\in\Omega^1(L)$ is not exact. So the action class 
$\alpha:=[\iota^*\lambda]\in H^1(L,\R)$ is nontrivial. Polterovich \cite{P} proved that 
the Maslov class $\mu\in H^1(L,\Z)$ has divisibility 2, i.e., $2=\min \{ \mu(\gamma)\mid \gamma \in H_1(L,\Z), \ \mu(\gamma)>0\}$. 

\vspace{.1in} 
\noindent 
{\bf Clifford tori.} \  
For $a,b>0$ the {\em Clifford torus}  
\[ 
T_{a,b}:=\{ |z_1|=a,\  |z_2|=b\} \subset \R^4 
\] 
is Lagrangian. It is monotone iff $a=b$.

Using symplectic capacities introduced by  Ekeland and  Hofer, Chekanov \cite{C1} proved the following: 

\begin{prop}[Chekanov \cite{C1}] \label{Tab}   % prop  \label{Tab} 
Two Clifford tori $T_{a,b},T_{a',b'}$ are Hamiltonian isotopic iff $T_{a',b'}=T_{a,b}$ or $T_{b,a}$. 
\end{prop}

\begin{lem}  \label{Tb}   %\label{Tb} 

The Hamiltonian monodromy group $\cH_b$ of $T_{b,b}$ is a group of order 2, i.e., it is generated by 
a single involution, hence $t(T_{b,b})=0$. Moreover $s(T_{b,b})=2$. 
\end{lem} 

\begin{proof} 
We take $\gamma\in H_1(T_{b,b},\R)$ to be the class represented by the curve 
$\{ (be^{i\theta},b)\in \C\times \C\}\mid \theta\in [0,2\pi]\}$. We also take $\gamma' \in H_1(T_{b,b},\R)$ 
 to be the class represented by the curve 
 $\{ (b,be^{i\theta})\in \C\times \C\}\mid \theta\in [0,2\pi]\}$. It can be checked that $\mu(\gamma)=2=\mu(\gamma')$. 
With this understood we 
 adapt the notations $\tilde{f}_k,\tilde{g}_k$ for monodromies from Remark 
 \ref{tilde} in Section \ref{mono}.

We identify $T_{b,b}$ with $L=\R/\Z\times \R/\Z$ with coordinates $(t_1,t_2)$ 
so that, for $t\in \R/\Z$, 
\begin{equation}  \label{t12} % equation \label{t12} 
\text{$\{ (t,0)\}$ represents the class $\gamma$, and $\{ (0, t)$\}  represents the class $\gamma'$}. 
\end{equation}  
Also let $(s_1,s_2)$ be the dual coordinates for fibers of the cotangent bundle $T^*L$. 
The cotangent bundle $(T^*L,-d\lac))$ is symplectic, where $\lac$ is the canonical 
1-form (see \cite{MS}), $\lac=s_1dt_1+s_2dt_2$. 
We also use the identification $\R^4\cong \C^2=\{ (r_1e^{\sqrt{-1}\theta_1},r_2e^{\sqrt{-1}\theta_2})\mid r_i\geq 0, \ \theta_i\in \R/2\pi \Z\}$. 

Now consider the map $\Phi:T^*L\to \R^4\cong \C^2=\{ (r_1e^{\sqrt{-1}\theta_1},r_2e^{\sqrt{-1}\theta_2})\} $, 
\[ 
\Phi(t_1,t_2,s_1,s_2):=\Big(\sqrt{b^2-\frac{s_1}{\pi}}e^{2\pi\sqrt{-1}t_1}, 
\sqrt{b^2-\frac{s_2}{\pi}}e^{2\pi\sqrt{-1}t_2}\Big). 
\] 
The map $\Phi$ is defined on the domain $U_L:=\{ s_1<\pi b^2,s_2<\pi b^2\}$, and is a symplectic embedding from $U_L$ into $\R^4$, $T_{b,b}\subset \Phi(U_L)$. 

Consider the primitive 1-form $\lambda=\frac{1}{2}(r^2_1d\theta_1+r^2_2d\theta_2)$ 
of $\omega$. We have $\Phi^*\lambda=\pi b^2(dt_1+dt_2)-\lac$. Also, 
Let $ L_{c_1,c_2}:=\{ s_1=c_1,s_2=c_2\}\subset U_L$, then 
\begin{equation} \label{Lc1c2} % equation \label{Lc1c2}
\Phi(L_{c_1,c_2})=T_{\sqrt{b^2-\frac{c_1}{\pi}},\sqrt{b^2-\frac{c_2}{\pi}}}. 
\end{equation} 

In the following we use $\Phi$ to identify a small neighborhood of $T_{b,b}$ with 
$U_\delta:=\{ |s_1|<\delta, |s_2|<\delta\}$ for $\delta>0$ small.

We claim that  $\tilde{g}_k\not\in \cH_b$ for any $k\neq 0$. Assume in the 
contrary that $\tilde{g}_k\in \cH_b$ for some $k\neq 0$. 
Let $\phi\in \Ham(\R^4,T_{b,b})$ be one with $\phi_*=g_k$. 
Modifying $\phi$ by a $L$-preserving Hamiltonian isotopy if necessary, we may 
assume that, on $U_\delta$ for some $\delta>0$,  
\[ 
\phi (t_1,t_2,s_1,s_2)=((1-k)t_1-kt_2,kt_1+(1+k)t_2, (1+k)s_1-ks_2,ks_1+(1-k)s_2).  
\] 
Then, by taking $c_1=0$ and $c_2=\epsilon>0$ very small, we have 
\[ 
\phi(L_{0,\epsilon})=L_{-k\epsilon,(1-k)\epsilon}. 
\] 
and hence (via $\Phi$) 
\[ 
\phi(T_{b,\sqrt{b^2-\frac{\epsilon}{\pi}}})=T_{\sqrt{b^2-\frac{-k\epsilon}{\pi}},\sqrt{b^2-\frac{(1-k)\epsilon}{\pi}}}. 
\] 
It then implies that, for all $\epsilon>0$ small enough, the Clifford tori 
$T_{b,\sqrt{b^2-\frac{\epsilon}{\pi}}}$ and $T_{\sqrt{b^2-\frac{-k\epsilon}{\pi}},\sqrt{b^2-\frac{(1-k)\epsilon}{\pi}}}$ are Hamiltonian isotopic, 
which cannot be possible by Proposition \ref{Tab}, unless $k=0$. Thus $\tilde{g}_k\not\in \cH_b$ for any $k\neq 0$. Hence $t(T_{b,b})=0$. 

Note that since $\tilde{f}_k\tilde{f}_l=\tilde{g}_{k-l}$, $\cH_b$ can contain at most one involution. In fact there exists a 
Hamiltonian self-isotopy of $T_{b,b}$ with monodromy $\tilde{f}_1=\begin{pmatrix} 0 & 1\\ 
1 & 0\end{pmatrix}$. To see this, first let us consider the 
path in the unitary group $U(2)$ defined by 
\[ 
A_t:=\begin{pmatrix} \cos \frac{\pi  t}{2} & -\sin \frac{\pi  t}{2} \\ \sin \frac{\pi  t}{2} & \cos \frac{\pi  t}{2} \end{pmatrix}\in GL(2,\C), \quad 0\leq t\leq 1. 
\] 
$A_t$ acts on $\C^2$, is the time $t$ map of the Hamiltonian vector field 
$X=\frac{\pi}{2}(x_1\partial_{x_2}-x_2\partial_{x_1}+y_1\partial_{y_2}-y_2\partial_{y_1})$, $\omega(X, \cdot )=-dH$, $H=\frac{\pi}{2}(x_2y_1-x_1y_2)$. Observe that $A_1(T_{a,b})=T_{b,a}$,  $(A_1)_*=\tilde{f}_1$ on $H_1(T_{b,b},\Z)$.
Fix $b>0$ and modify $H$  to get a $C^\infty$ function $\tilde{H}$ with compact support such that $\tilde{H}=H$ on $\{ |z_1|\leq 2b, \ |z_2|\leq 2b\}$. Let $\phi_t$ be the time $t$ map of the flow of the Hamiltonian vector field associated to $\tilde{H}$. Then $\phi_1(T_{b,b})=(T_{b,b})$, and $(\phi_1)_*=(A_1)_*=\tilde{f}_1$ on $H_1(T_{b,b},\Z)$. Hence $\cH_b$ is a group of order $2$ generated by the involution $\tilde{f}_1$. Hence $s(T_{b,b})=2$ by Remark \ref{tilde}, 
Proposition \ref{012} and Remark \ref{f0f1}. 
\end{proof}

\vspace{.1in} 
\noindent 
{\bf Chekanov tori.} \ 
Now we  consider another type of monotone Lagrangian tori in $\R^4$: the Chekanov tori (called {\em special} tori in \cite{C1}).  Consider the diffeomorphism $\rho:T^*S^1=S^1\times \R\to E:=\R^2_{x_1,x_2}\setminus \{ (0,0)\}$ defined by $\rho(\theta,s)=(e^s \cos\theta, e^s\sin\theta)$. 
The  corresponding map $\rho^*:T^*E\subset \R^4\to T^*(T^*S^1)$ is a symplectomorphism between 
two cotangent bundles. Let 
\[
\Psi:=(\rho^*)^{-1}:T^*(T^*S^1)=(T^*S^1)\times \R^2\to T^*E=E\times \R^2_{y_1,y_2}\subset \R^4
\]
 be the inverse symplectic map. 
 Let $(\theta,s)\in S^1\times \R$ be coordinates for $T^*S^1$, $(\theta^*,s^*)$ be 
 the dual coordinate for the fiber of $T^*(T^*S^1)$. 
Let $(x_1,x_2,y_1,y_2)$ be coordinates for $E\times \R^2$. 
Then 
\[ 
\Psi(\theta,s,\theta^*,s^*)=(e^s\cos\theta, e^s\sin\theta, e^{-s}(-\theta^*\sin\theta
+s^*\cos\theta),e^{-s}(\theta^*\cos\theta+s^*\sin\theta)). 
\] 
 
From now on, we will identify $T^*(T^*S^1)$ with its image in $\R^4$ via $\Psi$. 

For $a\in \R$ and $b>0$ the torus 
\[ 
T'_{a,b}:=\{ \theta^*=a, \ s^2+(s^*)^2=b^2\} 
\] 
is Lagrangian. Moreover, for $a\neq 0$, $T'_{a,b}$ is Hamiltonian isotopic to the 
Clifford torus $T_{b,b+|a|}$ by Chekanov \cite{C1}. The $a=0$ case is special. We call the special torus 
$T'_{0,b}$ a {\em Chekanov torus}. $T'_{0,b}$ is monotone, is Lagrangian isotopic to 
$T_{b,b}$ but not Hamiltonian isotopic to $T_{b,b}$.  \cite{C1,C2}.

\begin{lem}  \label{T'b}   %\label{T'b}
The Hamiltonian monodromy group $\cH'_b$ of $T'_{0,b}$ is a group of order 2, i.e., it is generated by 
a single involution, hence $t(T'_{0,b})=0$. Moreover $s(T'_{0,b})=1$. 
\end{lem}

\begin{proof} 
Let $\gamma\in H_1(T'_{0,b},\Z)$ be represented by the curve 
$(\theta,s,\theta^*,s^*)=(0,b\cos t,0,b\sin t)$, $\sigma \in H_1(T'_{0,b},\Z)$ be represented by the curve 
$(\theta,s,\theta^*,s^*)=(t,b,0,0)$. It can be verified that $\mu(\gamma)=2$ and $\mu(\sigma)=0$. With this understood we adapt the notations $f_k,g_k$ for monodromies from the proof of Proposition \ref{Gmu}.

We identify $T_{b,b}$ with $L=\R/\Z\times \R/\Z$ with coordinates $(t_1,t_2)$ 
so that, for $t\in \R/\Z$, 
\begin{equation}  \label{t'12} % equation \label{t'12} 
\text{$\{ (t,0)\}$ represents the class $\gamma$, and $\{ (0, t)$\}  represents the class $\sigma$}. 
\end{equation}  
Also let $(s_1,s_2)$ be the dual coordinates for the fiber of $T^*L$. . 
The cotangent bundle $(T^*L,-d\lac))$ is symplectic, where $\lac$ is the canonical 
1-form (see \cite{MS}), $\lac=s_1dt_1+s_2dt_2$.

Consider the map $\Phi':T^*L\to T^*(S^1\times \R)=\{ (\theta, s,\theta^*,s^*)\}$, 
\[ 
\Phi'(t_1,t_2,s_1,s_2):=(t_2,\sqrt{b^2-\frac{s_1}{\pi}}\cos 2\pi t_1, s_2, \sqrt{b^2-\frac{s_1}{\pi}}\sin 2\pi t_1).  
\] 
The map $\Phi'$  is defined on the domain $U'_L:=\{ s_1<\pi b^2\} \subset T^*L$, and is a symplectic embedding  from $U'_L$ into $T^*(S^1\times \R)$.  Also $T'_{0,b}\subset \Phi'(U'_L)$. 

Let $ L_{c_1,c_2}:=\{ s_1=c_1,s_2=c_2\}\subset U'_L$, then for $c_2\neq 0$, 
\begin{equation} \label{L'c1c2} % equation \label{L'c1c2}
\Psi\circ \Phi'(L_{c_1,c_2})=\Psi\Big(T'_{c_2,\sqrt{b^2-\frac{c_1}{\pi}}}\Big)=
T_{\sqrt{b^2-\frac{c_1}{\pi}},\sqrt{b^2-\frac{c_1}{\pi}}+|c_2|}, 
\end{equation} 
where the second equality is up to a Hamiltonian isotopy.

Assume there exists $\phi\in\Ham(\R^4,T'_{0,b})$ with monodromy $g_k$ for some $k\in \Z$. 
Modifying $\phi$ by a $L$-preserving Hamiltonian isotopy if 
necessary we may assume that, for some $\delta>0$ small enough, 
\[ 
\phi (\theta_1,\theta_2,s_1,s_2)=(\theta_1,k\theta_1+\theta_2, s_1-ks_2,s_2) \quad \text{on $U'_\delta$}.  
\] 
Then $\phi(L_{c_1,c_2})=L_{c_1-kc_2,c_2}$. 

Now, by taking $c_1=c_2=\epsilon>0$ very small, we have 
\[ 
\phi(L_{\epsilon,\epsilon})=L_{(1-k)\epsilon,\epsilon}. 
\] 
and hence (via $\Psi\circ \Phi'$, see (\ref{L'c1c2})) 
\[ 
\phi(T_{\sqrt{b^2-\frac{\epsilon}{\pi}},\sqrt{b^2-\frac{\epsilon}{\pi}}+\epsilon})=T_{\sqrt{b^2-\frac{(1-k)\epsilon}{\pi}},\sqrt{b^2-\frac{(1-k)\epsilon}{\pi}}+\epsilon}, 
\] 
It then implies that the Clifford tori 
$T_{\sqrt{b^2-\frac{\epsilon}{\pi}},\sqrt{b^2-\frac{\epsilon}{\pi}}+\epsilon}$ and 
$T_{\sqrt{b^2-\frac{(1-k)\epsilon}{\pi}},\sqrt{b^2-\frac{(1-k)\epsilon}{\pi}}+\epsilon}$  are Hamiltonian isotopic for all $\epsilon>0$ small enough, 
which cannot be possible by Proposition \ref{Tab}, unless $k=0$. Thus $g_k\not\in \cH'_b$ for any $k\neq 0$. Hence $t(T'_{0,b})=0$.

Then $\cH'_b$ can contain at most one involution. Below we will construct  $\phi\in \Ham(\R^4,T'_{0,b})$ 
with monodromy $\phi_*=f_0=\begin{pmatrix}1 & 0\\ 0& -1\end{pmatrix}$. First observe that $T'_{0,b}$ is contained in the hyper-surface 
$\{ \theta^*=0\} \subset T^*(S^1\times \R)$ and hence in (via $\Psi$) 
\begin{equation} \label{theta*=0} %\label{theta*=0} 
\{ (x_1,y_1,x_2,y_2)=(e^s\cos\theta, e^{-s}s^*\cos\theta,e^s\sin\theta, e^{-s}s^*\sin\theta)\} \subset \R^4. 
\end{equation} 
For $t\in [0,1]$ the symplectomorphism $A_t:\C^2\to \C^2$, $A_t(z_1,z_2)=(e^{i\pi t}z_1,z_2)$,   
is the time $t$ map of the Hamiltonian vector field $X=\pi(x_1\partial_{y_1}-y_1\partial_{x_1})$ 
whose Hamiltonian function is $H=\frac{\pi}{2}|z_1|^2$. 
Observe that $A_1$ preserves the hyper-surface in (\ref{theta*=0})  and send the point in (\ref{theta*=0}) 
to the point 
\begin{equation} 
\begin{split} 
(x_1,y_1,x_2,y_2) & =(-e^s\cos\theta,- e^{-s}s^*\cos\theta,e^s\sin\theta, e^{-s}s^*\sin\theta) \\ 
 & = (e^s\cos(\pi -\theta),e^{-s}s^*\cos(\pi -\theta),e^s\sin(\pi -\theta), e^{-s}s^*\sin(\pi -\theta) ).  
\end{split}
\end{equation}  
Hence $A_1(T'_{0,b})=T'_{0,b}$.  Moreover, $(A_1)_*\gamma=\gamma$ and $(A_1)_*\sigma=-\sigma$, 
i.e., $(A_1)_*=f_0$ on $H_1(T'_{0,b},\Z)$. 

We modify $H$ to get 
$\tilde{H}\in C^\infty(\R^4)$ with compact support, such that $\tilde{H}=H$ on $\{ |z_1|\leq e^{2b}, \ 
|z_2|\leq e^{2b}\}$. Let $\phi_t$ be the time $t$ map of the flow of the Hamiltonian vector field 
associated to $\tilde{H}$. Then $\phi_1(T'_{0,b})=T'_{0,b}$ and $(\phi_1)_*=(A_1)_*=f_0$ on 
$H_1(T'_{0,b},\Z)$. So $\cH_b$ is generated by the involution $f_0$ and $s(T'_{0,b})=1$ by 
Proposition \ref{012}, Remark \ref{f0f1} and Definition \ref{st}. 
\end{proof} 

Lemma \ref{Tb} and Lemma \ref{T'b} together imply Theorem \ref{main}.

\vspace{.2in} 
\noindent 
{\bf Final discussion.} \  
We end this note with the following open questions. 

\begin{ques} 
{\rm 
Let $L\subset \R^4$ be any monotone Lagrangian torus. Is it true that $\cH_L\cong \Z_2$? 
} 
\end{ques} 

\begin{ques} 
{\rm 
Let $L\subset (\R^4,\omega)$ be either a monotone Clifford torus $T_{b,b}$ or a Chekanov torus $T'_{0,b}$. Let $B\subset \R^4$ be an open  4-ball containing $L$. Assume there is a symplectic embedding 
$\phi:(B,\omega)\to (M,\omega_M)$, where $M$ is symplectic with $c_1(M)=0$ and $H^1(M,\R)=0$.  Let $\cH^B$ (resp. $\cH^M$) denote the Hamiltonian monodromy group of $L$ in $B$ (resp. in $M$). Via the inclusion $\phi$, $\cH^B\cong \Z_2$ is a subgroup of $\cH^M$. Is it possible that $\cH^B$ is a proper subgroup of $\cH^M$? 
}
\end{ques}

 \begin{ques} 
 {\rm 
 How to extend the constructions of $t(L),s(L)$ to higher dimensional cases, to distinguish monotone Lagrangian tori in $\R^{2n}$ and beyond? 
 } 
 \end{ques}

\section*{Acknowledgements}
The author thanks Liang-Chung Hsia for references on reflection groups. 
The author also thanks an anonymous referee for pointing out typos and minor mistakes in an earlier draft of this paper.

\vspace{.1in } 
Department of Mathematics, National Central University,  Chung-Li, Taiwan. 

{\em Email address}: yau@math.ncu.edu.tw

\end{document}